\documentclass[smallcondensed, envcountsect]{svjour3}       
\smartqed  
\usepackage{graphicx}
\usepackage[misc]{ifsym}

\usepackage{fixmath}
\usepackage{psfrag, subfigure, bm}
\usepackage[usenames,dvipsnames]{pstricks}
\usepackage{amsmath,amsfonts,amssymb}
\usepackage{mathabx}
\usepackage[numbers, sort&compress]{natbib}
\usepackage[a4paper, inner=3cm, outer=3cm, bottom=3cm, top=3cm]{geometry}

\usepackage{latexsym}
 \newcommand{\myset}{ }
 \renewcommand{\vec}{\mathrm{vec}}
 \newcommand{\transp}{\mathsf{T}}
 \newcommand{\tr}{\mathrm{tr}}
 \newcommand{\x}{\mathbold{x}}
 \newcommand{\uu}{\mathbold{u}}
 \newcommand{\X}{\mathbold{X}}
 \newcommand{\Y}{\mathbold{Y}}
 \renewcommand{\d}{\mathbold{d}}
 \newcommand{\g}{\mathbold{g}}
 \newcommand{\zv}{\mathbold{z}_{\mathrm{sv}}}
 \newcommand{\y}{\mathbold{y}}
 \renewcommand{\v}{\mathbold{v}}
 \newcommand{\z}{\mathbold{z}}
 \newcommand{\A}{\mathbold{A}}
 
 \newcommand{\G}{\mathbold{G}}
 \newcommand{\Q}{\mathbold{Q}}
 \newcommand{\h}{\mathbold{h}}
 \newcommand{\hv}{\mathbold{h}_{\mathrm{sv}}}
 \newcommand{\Zv}{Z_{\mathrm{sv}}}
 \newcommand{\W}{\mathbold{W}}
 \newcommand{\subjectto}{\mathrm{subject~to}}
 \DeclareMathOperator*{\minimize}{\mathrm{minimize}}
 
 \DeclareMathOperator*{\argmin}{\mathrm{argmin}}

 \newtheorem{assumption}{Assumption}[section]
 
\begin{document}

\title{Primal Recovery from Consensus-Based Dual Decomposition for Distributed Convex Optimization
}

\titlerunning{Primal Recovery from Consensus-Based Dual Decomposition}        

\author{Andrea Simonetto         \and
        Hadi Jamali-Rad 
}


\institute{A. Simonetto (\Letter) \and H. Jamali-Rad \at
              Faculty of Electrical Engineering, Mathematics, and Computer Science, \\
              Delft University of Technology, 2826 CD Delft, The Netherlands. \\
              Tel.: +31-1527-82845\\
              Fax: +31-1527-86190\\
              \email{$\{$a.simonetto, h.jamalirad$\}$@tudelft.nl}           
}

\date{Received: date / Accepted: date}

\maketitle

\begin{abstract}
Dual decomposition has been successfully employed in a variety of distributed convex optimization problems solved by a network of computing and communicating nodes. Often, when the cost function is separable but the constraints are coupled, the dual decomposition scheme involves local parallel subgradient calculations and a global subgradient update performed by a master node. In this paper, we propose a consensus-based dual decomposition to remove the need for such a master node and still enable the computing nodes to generate an approximate dual solution for the underlying convex optimization problem. In addition, we provide a primal recovery mechanism to allow the nodes to have access to approximate near-optimal primal solutions. Our scheme is based on a constant stepsize choice and the dual and primal objective convergence are achieved up to a bounded error floor dependent on the stepsize and on the number of consensus steps among the nodes. 


\vskip0.2cm

\end{abstract}
\newpage
\section{Introduction}
\label{intro}

Lagrangian relaxation and dual decomposition are extremely effective in solving large-scale convex optimization problems~\cite{Bertsekas1999,Johansson2008a,Boyd2008,Nedic2009a,Polyak1987,Kiwiel2004}. Dual decomposition has also been employed successfully in the field of distributed convex optimization, where the optimization problem requires to be decomposed among cooperative computing entities (called in the following simply by nodes). In this case, the optimization problem is generally divided into two steps, a first step pertaining the calculation of the local subgradients of the Lagrangian dual function, and a second step consisting of the global update of the dual variables by projected subgradient ascent. The first step can typically be performed in parallel on the nodes, whereas the second step has often to be performed centrally, by a so-called \emph{master} node (or data-gathering node, or fusion center), which combines the local subgradient information. 

Even though by solving the dual problem, one obtains a lower bound on the optimal value of the original convex problem, in practical situations one would also like to have access to an approximate primal solution. However, even with the availability of an approximate dual optimal solution, a primal one cannot be easily obtained. The reason is that the Lagrangian dual function is generally nonsmooth at an optimal point, thus an optimal primal solution is not a trivial combination of the extreme subproblem solutions. Methods to recover approximate (near-optimal) primal solutions from the information coming from dual decomposition have been proposed in the past~\cite{Nemirovskii1978,Shor1985,Sherali1996,Larsson1999,Ma2007,Nedic2009a,Gustavsson2014,Necoara2014} (and references therein). In one way or another, all these methods use a combination of all the approximate primal solutions that are generated while the dual decomposition scheme converges to a near-optimal dual solution. A possible choice for the combination is the ergodic mean~\cite{Ma2007,Nedic2009a,Nedic2009}. 

Among the dual decomposition schemes with primal recovery mechanism available in the literature, we are interested here in the ones that employ a constant stepsize in the projected dual subgradient update. The reasons are twofold. First of all, a constant stepsize yields faster convergence to a bounded error floor, which is fundamental in real-time applications (e.g., control of networked systems). In addition, the error floor can be tuned by trading-off the number of iterations required and the value of the stepsize. The second reason is that in many situations the underlying convex optimization problem is not stationary, but changes over time. Having in mind the development of methods to update the dual variables while the optimization problem varies~\cite{Jakubiec2013,Simonetto2014c,Simonetto2014b}, it is of key importance to employ a constant stepsize. In this way, the capability of the subgradient scheme to track the dual optimal solutions does not change over time due to a vanishing stepsize approach. 

In this paper, we propose a way to remove the need for a \emph{master} node to collect the local subgradient information coming from the different nodes and generate a global subgradient. The reason is that in distributed systems, the nodes are connected via an ad-hoc network and the communication is often limited to geographically nearby nodes. It is therefore impractical to collect the local subgradient information in one physical location, whereas it is advisable to enable the nodes themselves to have access to a suitable approximation of the global subgradient. We use consensus-based mechanisms to construct such an approximation. Consensus-based mechanisms have been used in the primal domain both with constant stepsizes~\cite{Johansson2008,Srivastava2011} and with vanishing ones~\cite{Nedic2010,Srivastava2011,Duchi2012}, however, to the best of the authors' knowledge, they have not been used in the dual domain, and not together with primal recovery. An interesting, but different, approach applying consensus on the cutting-plane algorithm to solve the master problem has been very recently proposed in~\cite{Burger2014}. Our main contributions can be described as follows.

First, we develop a constant stepsize consensus-based dual decomposition. Our method enables the different nodes to generate a sequence of approximate dual optimal solutions whose dual cost eventually converges to the optimal dual cost within a bounded error floor. Under the assumptions of convexity, compactness of the feasible set, and Slater's condition, the convergence goes as $O(1/k)$, where $k$ is the number of iterations. The error depends on the stepsize and on the number of consensus steps between subsequent iterations $k$. Furthermore, the nodes are exchanging subgradient information only with their nearby neighboring nodes. 

Then, since in our method, each node maintains its own approximate dual sequence, we provide an upper bound on the disagreement among the nodes, and we prove its convergences to a bounded value.
 
Finally, we propose a primal recovery scheme to generate approximate primal solutions from consensus-based dual decomposition. This scheme is proven to converge to the optimal primal cost up to a bounded error floor. Once again, under the same assumptions, the convergence goes as $O(1/k)$ and the error depends on the stepsize and on the number of consensus steps. 

 
\vskip0.2cm
\noindent\textbf{Organization.} Section~\ref{sec:2} describes the problem setting, our main research question, and some sample applications. In Section~\ref{sec:3}, we cover the basics of dual decomposition to pinpoint the main limitation, i.e., the need for a \emph{master} node. We propose, develop, and investigate the convergence results of our algorithm in Sections~\ref{sec:4} and \ref{sec:5}. All the proofs are contained in Sections~\ref{sec:6} and \ref{sec:7}. In Section~\ref{sec:num}, we collect numerical simulation results. Future research questions and conclusions are discussed in Sections~\ref{sec:rquestion} and \ref{sec:8}, respectively.

\section{Problem Formulation} \label{sec:2}

\noindent\textbf{Notation.} For any two vectors $\x,\y \in \mathbb{R}^n$, the standard inner product is indicated as $\langle \x, \y \rangle$, while its induced (Euclidean) norm is represented as $\|\x\|_{2}$. A vector $\x$ belongs to $\mathbb{R}^n_+$ iff it is of size $n$ and all its components are nonnegative (i.e., $\mathbb{R}^n_+$ is the nonnegative orthant). For any vectors $\x\in\mathbb{R}^n$, its components are indicated by $x_i$, $i\in\{1, \dots, n\}$. The vector $\mathbf{1}_n$ is the column vector of length $n$ containing only ones. We indicate by ${\mathbf{I}}_n$ the identity matrix of size $n$. For any real-valued squared matrix $\X\in\mathbb{R}^{n\times n}$, we say $\X \succeq 0 $ or $\X \preceq 0$ iff the matrix is positive semi-definite or negative semi-definite, respectively. We also write $\X\in \mathbb{S}_{+}^n$, iff $\X \succeq 0$. For any real-valued squared matrix $\X\in\mathbb{R}^{n\times n}$, the norm $\|\X\|_{\mathrm{F}}$ represents the Frobenius norm, while the trace is indicated by $\tr[\X]$. %
The symbol $(\cdot)^\transp$ is the transpose operator, $\otimes$ represents the Kronecker product, $\circ$ stands for map composition, $\mathrm{conv}[\cdot]$ is the convex hull, $\vec(\cdot)$ is the vectorization operator, while $\mathsf{P}_{{X}}[\cdot]$ is the projection operator onto the set $X$.  
The $\epsilon$-subgradient of a \emph{concave} function $q(\x): \myset{X}\subseteq \mathbb{R}^n \to \mathbb{R}$, for the non-negative scalar $\epsilon\geq 0$, at $\x'\in X$ is a vector $\tilde{\g}\in\mathbb{R}^n$ such that
\begin{equation}\label{egrad}
\langle \tilde{\g}, \y-\x' \rangle \geq q(\y) - q(\x') - \epsilon, \quad \forall {\y} \in \myset{X}.
\end{equation}
Furthermore, the collection of  $\epsilon$-subgradients of $q(\x)$ at $\x'$ is called the $\epsilon$-subdifferential set, denoted by $\partial_{\epsilon} q_\x(\x')$. If $\epsilon = 0$ the $\epsilon$-subgradient is the regular subgradient and we drop the $\epsilon$ in the notation of the subdifferential. 

\vskip0.2cm

\noindent\textbf{Formulation.} We consider a convex optimization problem defined on a network of computing and communicating nodes. Let the nodes be labeled with $i\in\myset{V} = \{1,\dots,n\}$ and we equip each of them with the local (private) convex function $f_i(x_i):\mathbb{R} \to \mathbb{R}$. Let $\x$ be the stacked vector of all the local decision variables, i.e., $\x = (x_1, \dots, x_n)^\transp$.
Let the functions $g_{i}(x_i): \mathbb{R} \to \mathbb{R}, i\in\myset{V}$ be convex. Let $\A_{0}, \A_{i}, i\in\myset{V}$ be $d\times d$ real-valued square and symmetric matrices. Let $X_i \subset \mathbb{R}, i\in\myset{V}$ be convex and \emph{compact} sets, and let $X := \prod_{i\in\myset{V}} X_i$. We are interested in solving \emph{decomposable} convex optimization problems of the form, 
\begin{subequations}\label{opt.prob}
\begin{align}
\minimize_{x_i \in \myset{X}_i, i\in\myset{V} } &\quad f(\x):= \sum_{i\in\myset{V}} f_i(x_i) \\
\subjectto &\quad \sum_{i\in\myset{V}} g_{i}(x_i) \leq 0, \label{eq.c1}\\ 
&\quad \A_{0} + \sum_{i\in\myset{V}} \A_{i} x_i \succeq 0.\label{eq.c2}
\end{align}
\end{subequations}
In order to simplify our notation (and without loss of generality) we have chosen to work with scalar decision variables $x_i$, with one scalar inequality, and with one linear matrix inequality. The following assumptions are in place. 

\begin{assumption} \label{as.0} \emph{(Convexity and compactness)}
The cost functions $f_i(x_i)$ and the constraint functions $g_i(x_i)$ are convex in $x_i$ for each $i$. The sets $X_i$ are convex and compact (thus, bounded). The matrices $\A_{0}, \A_{i}, i\in\myset{V}$ are real-valued square and symmetric.
\end{assumption}

\begin{assumption} \label{as.ex} \emph{(Existence of solution)}
The feasible set 
$
F: = \{\x \in X| \eqref{eq.c1} $  and $ \eqref{eq.c2}\}
$ 
is nonempty; for all $\x\in F$ the cost function $f(\x)>-\infty$, and there exists a vector $\x\in F$ such that $f(\x)<\infty$.  
\end{assumption}

\begin{assumption} \label{as.slater} \emph{(Slater condition)}
There exists a vector $\bar{\x} \in \mathbb{R}^n$ that is strictly feasible for problem~\eqref{opt.prob}, i.e., 
\begin{equation*}
\sum_{i\in\myset{V}} g_{i}(\bar{x}_i) < 0,\,\textrm{and} \quad \A_{0} + \sum_{i\in\myset{V}} \A_{i} \bar{x}_i \succ 0.
\end{equation*}
\end{assumption}

\begin{assumption} \label{as.comm} \emph{(Communication network)}
The computing nodes communicate synchronously via undirected time-invariant communication links.
\end{assumption}

Assumption~\ref{as.0} is required to ensure a convex program with compact feasible set. Assumption~\ref{as.ex} ensures the existence of a solution for the optimization problem~\eqref{opt.prob}. Let $\x^*$ be such a (possibly not unique) solution (i.e., a minimizer) and let $f^*$ be the unique minimum. Assumption~\ref{as.slater} is often required in dual decomposition approaches in order to guarantee \emph{zero duality gap} and to be able to derive the optimal value of the optimization problem~\eqref{opt.prob} by solving its dual. In addition, Slater condition helps in bounding the dual variables, which is crucial in our convergence analysis. Assumption~\ref{as.comm} is required to simplify the convergence analysis. One might be able to loosen it and require only asynchronous communications, but this is left for future research since it is not the core idea of this paper. By Assumption~\ref{as.comm}, we can define an undirected communication graph $\mathcal{G}$ consisting of a vertex set $\myset{V}$ as well as an edge set $\myset{E}$. For each node $i$, we call \emph{neighborhood}, or $N_i$, the set of the nodes it can communicate with. 

The main research problem we tackle in this paper can be stated as follows. 
\vskip0.2cm
\noindent \textbf{Research problem:} \emph{we would like to devise an algorithm that enables each node, by communicating with their neighbors only, to construct a sequence of approximate local optimizers $\{x_i^k\}$, for which their primal objective sequence $\{f(\x^k)\}$ eventually converges to $f^*$ (possibly) up to a bounded error floor.} 

\vskip0.2cm

Our approach towards this problem is to devise a consensus-based dual decomposition with approximate primal recovery. 

\vskip0.2cm
\noindent\textbf{Sample applications.} Problems as~\eqref{opt.prob} appear in many contexts: the first example we cite is the network utility maximization (NUM) problem, where a group of communication nodes try to maximize their utility subject to a resource allocation constraint~\cite{Palomar2006,Wei2010}. NUM problems are very relevant in communication systems. Generalizations of NUM problems, where the cost function is separable and the variables are constrained by linear inequalities, can also be handled by~\eqref{opt.prob}, and have been considered, e.g., in model predictive controller design~\cite{Doan2012} (which is one of the workhorse of nowadays control theory). Another sample application is sensor selection, where a set of nodes try to decide which one of them should be activated to perform a certain task based on a given metric. This is in general a combinatorial problem, yet it can be relaxed to a semidefinite program, which is a generalization of~\eqref{opt.prob}, \cite{Jamali-Rad2014a, Joshi2009}. In the latter example, the constraint~\eqref{eq.c2} plays an important role. 

\vskip0.2cm
\noindent\textbf{Multi-agent/Multiuser/Networked problems. } If the constraints~\eqref{eq.c1} and \eqref{eq.c2} involve only local functions, that is the sum is only over the neighbors of a particular $i$, then we have what is known as multi-agent (or multiuser, or networked) problem. These problems can be further complicated by the presence of global decision variables. In all these cases, due to the presence of neighborhood constraint functions only, the dual variables associated to them can be computed locally in the neighborhood (we can refer to them as link dual variables). Therefore, by a proper use of dual decomposition, we can devise distributed algorithms that can be implemented on nodes and connecting links. Relevant recent work on these problems is reported in~\cite{Necoara2008,Nedic2011,Lu2013,Lu2013a,Necoara2014a,Simonetto2014,Beck2014,Simonetto2014d}. In our case, the constraints~\eqref{eq.c1}-\eqref{eq.c2} involve constraint functions from all the nodes, in all the decision variables together; therefore, the proposed methods for multi-agent problems cannot be directly applied in our case. In general, the link dual variables become a network-wide dual variable in our case, and we retrieve the standard dual decomposition scheme with the need for a master node to compute such a global network-wide dual variable.









\section{Dual Decomposition}\label{sec:3}

The Lagrangian function $L(\x, \mu, \G): \mathbb{R}^n \times \mathbb{R}_+ \times \mathbb{S}_+^d \to \mathbb{R}$ is formed, as a first step of dual decomposition, 
\begin{equation}
L(\x, \mu, \G) := \sum_{i\in\myset{V}} f_i(x_i) + \mu \Big(\sum_{i\in\myset{V}} g_i (x_i)\Big)  - \tr\Big[\Big(\A_{0} + \sum_{i\in\myset{V}} \A_{i} x_i\Big) \G \Big],
\end{equation}
where $\mu\in\mathbb{R}_+$ is the dual variable associated with the constraint~\eqref{eq.c1}, and $\G\in\mathbb{S}_+^d$ is the dual variable associated with~\eqref{eq.c2}. Further, the dual function $q(\mu, \G): \mathbb{R}_{+} \times \mathbb{S}_+^d \to \mathbb{R}$ can be defined as 
\begin{equation}\label{dual}
q(\mu, \G) := \min_{\x\in\myset{X}} \{L(\x, \mu, \G)\}.
\end{equation}
The set $\myset{X}$ is compact, which means that the function $q(\mu, \G)$ is continuous on $\mathbb{R}_{+} \times \mathbb{S}_+^d$. Furthermore, the function $q(\mu, \G)$ is concave. For any pair of dual variables $(\mu,\G)$, we can compute the value of the primal minimizers and their set: 
\begin{equation}
\tilde{\x} := \argmin_{\x\in\myset{X}} \{L(\x, \mu, \G)\}, \quad \tilde{X}:= \{{\x}\in\myset{X}| q(\mu, \G) = L({\x}, \mu, \G)\}.
\end{equation}
Given the compactness of $X$ and the form of the dual function~\eqref{dual}, we can define the subdifferential of $q(\mu, \G)$ at $\mu$ and $\G$ as the following sets 
\begin{subequations}
\begin{eqnarray}
\partial q_\mu(\mu, \G) &:=& \mathrm{conv}\Big[\sum_{i\in\myset{V}} g_i(\tilde{x}_i)| \tilde{\x} \in \tilde{X} \Big],\\
\partial q_\G(\mu, \G) &:=& \mathrm{conv}\Big[-\Big(\A_0 +\sum_{i\in\myset{V}} \A_{i} \tilde{x}_i\Big)| \tilde{\x} \in \tilde{X} \Big], 
\end{eqnarray}
\end{subequations}
Subgradient choices for $q(\mu, \G)$ are therefore 
\begin{equation}
h(\tilde{\x}) := \sum_{i\in\myset{V}} g_i(\tilde{x}_i) \in  \partial q_\mu(\mu, \G), \, \quad \Q(\tilde{\x}) := -\A_0 - \sum_{i\in\myset{V}} \A_{i} \tilde{x}_i \in  \partial q_\G(\mu, \G),
\end{equation}
for any choice of $\tilde{\x}\in\tilde{X}$. In addition, since $X$ is compact and the constraints~\eqref{eq.c1}-\eqref{eq.c2} are represented by continuous functions, the subgradients are bounded, and we set, for all $i\in\myset{V}$
\begin{equation}\label{LQ}
\|h_i({\x})\|_2 \leq \max_{x_i\in\myset{X}_i} \Big\| g_i(\x_i) \Big\|_2 =: L, \quad \|\Q_i(\x)\|_{\mathrm{F}} \leq \max_{x_i\in\myset{X}_i} \Big\| -\A_0/n- \A_{i} {x}_i \Big\| _{\mathrm{F}} =: Q, 
\end{equation}
where we have defined $h_i(\x): = g_i(\x_i)$, and $\Q_i(\x): = -\A_0/n- \A_{i} {x}_i$. Finally, the Lagrangian dual problem can be written as
\begin{equation}
q^* := \sup_{\mu \in \mathbb{R}, \G\in\mathbb{S}_+^d} \{q(\mu, \G)\},
\end{equation}
and by Slater condition (Assumption~\ref{as.slater}), strong duality holds: $q^* = f^*$.

Since the original convex optimization problem~\eqref{opt.prob} is decomposable, the Lagrangian function is separable as
\begin{equation}
L(\x, \mu, \G) \hskip-0.1cm=\hskip-0.1cm \sum_{i\in\myset{V}} \left( f_i(x_i) + {\mu} g_i (x_i)  -  \tr\Big[\Big( {\A_{0}}/{n} + \A_{i} x_i\Big) \G \Big]\right) =: \hskip-0.1cm \sum_{i\in\myset{V}} L_i(x_i, \mu, \G),
\end{equation}
and so is the dual function
\begin{equation}
q(\mu, \G) := \sum_{i\in\myset{V}} \min_{x_i\in\myset{X}_i} \{L_i(x_i, \mu, \G)\} := \sum_{i\in\myset{V}} q_i(\mu, \G),
\end{equation}
and its subgradients.

Dual decomposition with approximate primal recovery as defined in~\cite{Nedic2009a} is summarized in the following algorithm.

\newpage  

\noindent{\bf Dual decomposition with primal recovery} \hrule 
\begin{enumerate}
\item Initialize $\mu^0\in\mathbb{R}_+$, $\G^0\in\mathbb{S}_+^d$, choose a constant stepsize $\alpha$;
\item Local dual optimization: compute in parallel the local dual functions and their primal optimizers
\begin{subequations}\label{it.dds}
\begin{equation}
q_i(\mu^k, \G^k) = \min_{x_i\in\myset{X}_i} \{L_i(x_i, \mu^k, \G^k)\}, \quad \tilde{x}_i^k = \argmin_{x_i\in\myset{X}_i} \{L_i(x_i, \mu^k, \G^k)\},
\end{equation}
as well as their subgradients $g_i(\tilde{x}_i^k)$ and $-\A_0/n-\A_i \tilde{x}_i^k$;
\item Primal recovery step: compute in parallel the ergodic sum, for $k\geq 1$
\begin{equation}
x_i^k = \frac{1}{k} \sum_{t=1}^{k} \tilde{x}_i^t;
\end{equation}
\item Dual update: update the variables $\mu^k, \G^k$ as
\begin{eqnarray}
\mu^{k+1} &=& \mathsf{P}_{\mathbb{R}_+}\Big[\mu^{k} + \alpha \sum_{i\in\myset{V}} g_i(\tilde{x}_i^k)\Big]\\
\G^{k+1} &=& \mathsf{P}_{\mathbb{S}_+^d}\Big[\G^{k} - \alpha \Big(\A_0+\sum_{i\in\myset{V}}\A_i \tilde{x}_i^k\Big) \Big].
\end{eqnarray}
\end{subequations}
\end{enumerate}\hrule

\vskip0.2cm
This algorithm generates a converging sequence $\{x_i^k\}$ as detailed in the following theorem.  
\begin{theorem}\label{th.0}
Let the sequence $\{\mu^k, \G^k, \x^k\}$ be generated by the iterations in~\eqref{it.dds}. Let $L$ and $Q$ be defined as in~\eqref{LQ}. Under Assumptions~\ref{as.0} till \ref{as.slater}, 

\vskip0.1cm
\begin{minipage}{0.95\textwidth}
\begin{enumerate}
\item[(a)] the dual variables are bounded, i.e., $\|\mu^k\|_2 \leq \Lambda_0 < \infty$, $\|\G^k\|_{\mathrm{F}} \leq \Gamma_0 < \infty$, for all $k\geq 1$;
\item[(b)] an upper bound on the primal cost of the vector $\x^k$, $k\geq 1$, is given by
\begin{equation*}
f(\x^k)\leq f^* + \frac{\Lambda_0^2 + \Gamma_0^2}{2 \alpha k} + \frac{\alpha n^2 (L^2+Q^2)}{2};
\end{equation*}
\item[(c)] a lower bound on the primal cost of the vector $\x^k$, $k\geq 1$, is given by
\begin{equation*}
f(\x^k)\geq f^* - \frac{\Lambda_0^2 + \Gamma_0^2}{ \alpha k}.
\end{equation*}
\end{enumerate}
\end{minipage}
\end{theorem}

\begin{proof}
The proof follows from~\cite[Lemma~3 and Proposition~1]{Nedic2009a}. Since our optimization problem involves also a linear matrix inequality, some extra steps are needed in the proof of part \emph{(c)}. To be more specific, by following the same steps in the proof of~\cite[Proposition~1.\emph{(c)}]{Nedic2009a}, we arrive at the following inequality
\begin{equation}\label{dummy42}
f(\x^k) \geq f^* - \mu^* h(\x^k) - \tr[\Q(\x^k){\G^*}].
\end{equation}
where $\mu^*\geq 0$ and $\G^*\succeq 0$ are the optimal dual variables. We now need to find an lower bound for the rightmost term of~\eqref{dummy42}. By similar arguments of the proof of~\cite[Proposition~1.\emph{(a)}]{Nedic2009a}, we obtain for all $k \geq 1$ 
\begin{equation}\label{dummy41}
h(\x^k) \leq \frac{\mu^k}{\alpha k}, \qquad \frac{\G^k}{\alpha k} - \Q(\x^k)  \succeq 0.
\end{equation}
Given the two positive semi-definite matrices $\X$ and $\Y$ of dimension $n\times n$, we know that $\tr[\X\,\Y]\geq \lambda_{\min}(\X) \tr[\Y] \geq 0$, \cite[Lemma~1]{Wang1986}, which means 
\begin{equation*}
\tr\Big[\Big(\frac{\G^k}{\alpha k} - \Q(\x^k)\Big){\G^*}\Big] \geq 0, \quad\textrm{thus}\quad \tr\Big[\Big(\frac{\G^k}{\alpha k}\Big){\G^*}\Big] \geq \tr[{\Q(\x^k)}{\G^*}].
\end{equation*}
This implies that for $k\geq 1$
\begin{equation}\label{dummy40}
\tr[\Q(\x^k){\G^*}] \leq  \tr\Big[\frac{\G^k}{\alpha k} \G^*\Big] = \Big|\tr\Big[\frac{\G^k}{\alpha k} \G^*\Big] \Big| \leq \frac{1}{\alpha k} \|\G^k\|_{\mathrm{F}} \|\G^*\|_{\mathrm{F}} \leq \frac{\Gamma^2_0}{\alpha k},
\end{equation}
where we have used Cauchy-Schwarz inequality~\cite{Golub1996}. By combining~\eqref{dummy40} and \eqref{dummy41} with \eqref{dummy42}, we obtain the lower bound
$$
f(\x^k) \geq f^* - \mu^* h(\x^k) - \tr[\Q(\x^k){\G^*}] \geq f^* - \frac{\Lambda_0^2}{\alpha k} - \frac{\Gamma^2_0}{\alpha k},
$$
and the claim is proven. 
\qed
\end{proof}

Although, the dual decomposition method of \cite{Nedic2009a} presents several advantages, in practice, the nodes will need to sum the subgradients coming from the whole network in Step 4 in order to maintain common dual variables. This is often not practical in large networks, because it would call for a significant communication overhead.  

In the following sections, 
\emph{(i)} we propose a consensus-based dual decomposition with primal recovery mechanism to modify Step 4 in order to make it suitable for limited information exchange (i.e., communication only with neighboring nodes); 
\emph{(ii)} we prove dual and primal objective convergence of the proposed method up to a bounded error floor which depends (among other things) on the number of communication exchange with the neighboring nodes for each iteration $k$.

\section{Basic Relations}\label{sec:4}

\begin{lemma}\label{bound}
Suppose Assumption~\ref{as.0} till~\ref{as.slater} hold. Let $\bar{\mu} \geq 0$, $\bar{\G}\succeq 0$ be a pair of dual variables for which the set $\bar{D} := \{({\mu} \geq 0, {\G}\succeq 0)| q({\mu}, {\G}) \geq q(\bar{\mu}, \bar{\G}) \} $ is nonempty. Then, the set $\bar{D}$ is bounded and we have
\begin{equation*}
\max_{({\mu}, {\G}) \in \bar{D}}  \|\mu \|_2 + \|\G\|_{\mathrm{F}} \leq \frac{1}{\gamma} (f(\bar{\x}) - q(\bar{\mu}, \bar{\G})),
\end{equation*}
where $\gamma := \min\Big\{\sum_{i\in\myset{V}} -g_i(\bar{x}_i), \lambda_{\min}\big(\A_0 + \sum_{i\in\myset{V}} \A_i \bar{x}_i\big) \Big\}$, $\lambda_{\min}(\cdot)$ is the smallest eigenvalue and $\bar{\x}$ is a vector satisfying the Slater condition. 
\end{lemma}

\begin{proof}
The lemma follows from~\cite[Lemma~1]{Nedic2009a} with minor modifications. In particular, we use \cite[Lemma~1]{Wang1986} to bound the inner product
\begin{equation*}
\tr\Big[\Big(\A_0 + \sum_{i\in\myset{V}} \A_i \bar{x}_i \Big) \G^k \Big] \geq \lambda_{\min}\Big(\A_0 + \sum_{i\in\myset{V}} \A_i \bar{x}_i \Big) \tr[\G^k], 
\end{equation*}
and the fact that $\|\G\|_{\mathrm{F}} \leq \textrm{tr}[\G]$,~\cite{Golub1996}. The remaining steps are omitted since similar to~\cite[Lemma~1]{Nedic2009a}. \qed
\end{proof}
It follows from the result of the preceding lemma that under Slater, the dual optimal set $D^*$ is nonempty. Since $D^* := \{({\mu} \geq 0, {\G}\succeq 0)| q({\mu}, {\G}) \geq q^*\}$, by using Lemma~\ref{bound}, we obtain
\begin{equation*}
\max_{({\mu}^*, {\G}^*) \in D^*}  \|\mu^* \|_2 + \|\G^*\|_{\mathrm{F}} \leq \frac{1}{\gamma} (f(\bar{\x}) - q^*).
\end{equation*}
Furthermore, although the dual optimal value $q^*$ is not a priori available, one can compute a looser bound by computing the dual function for some couple $(\tilde{\mu} \geq 0, \tilde{\G}\succeq 0)$. Owning to optimality, $q^* \geq q(\tilde{\mu},\tilde{\G})$, thus
\begin{equation*}
\max_{({\mu}^*, {\G}^*) \in D^*}  \|\mu^* \|_2 + \|\G^*\|_{\mathrm{F}} \leq \frac{1}{\gamma} (f(\bar{\x}) - q(\tilde{\mu}, \tilde{\G})).
\end{equation*}
This result is quite useful to render the dual decomposition method easier to study. In fact, as in~\cite{Nedic2009a}, we can modify the sets over which we project in Step 4 by considering a bounded superset of the dual optimal solution set. This means that we can substitute Step 4 in~\eqref{it.dds} with 
\begin{subequations}\label{itbound}
\begin{eqnarray}
\mu^{k+1} &=& \mathsf{P}_{D_{\mu}}\Big[\mu^{k} + \alpha \sum_{i\in\myset{V}} g_i(\tilde{x}_i^k)\Big], \, D_{\mu} : = \Big\{\mu \geq 0 \,\Big|\, \|\mu\|_2 \leq \frac{f(\bar{\x}) - q(\tilde{\mu}, \tilde{\G})}{\gamma}  + r \Big\} \nonumber \\ && \\
\G^{k+1} &=& \mathsf{P}_{D_{\G}}\Big[\G^{k} - \alpha \Big(\A_0 + \sum_{i\in\myset{V}}\A_i \tilde{x}_i^k \Big)\Big], \nonumber \\ &&  \qquad \qquad \qquad \quad   D_{\G} : = \Big\{\G \succeq 0 \,\Big|\, \|\G\|_{\mathrm{F}} \leq \frac{f(\bar{\x}) - q(\tilde{\mu}, \tilde{\G})}{\gamma}  + r \Big\} 
\end{eqnarray}
\end{subequations}
for a given scalar $r > 0$. The nice feature of this modification is that both $D_{\mu}$ and $D_{\G}$ are now compact convex sets. This does not increase computational complexity, and it is a useful modification, for it provides a leverage to derive the convergence rate results. In the following, for convergence purposes, we will use $r \geq \frac{f(\bar{\x}) - q(\tilde{\mu}, \tilde{\G})}{\gamma}$.

\section{Consensus-Based Dual Decomposition}\label{sec:5}

We consider now a consensus-based update to enforce the update rule of dual decomposition in~\eqref{itbound} to fit the constraint of a limited communication network. Our approach is inspired by the one of \cite{Johansson2008} but applied to the dual domain. First of all, we define a consensus matrix $\W \in \mathbb{R}^{n\times n}$, with the following properties:
\begin{equation}\label{c.con}
[\W]_{ij} = 0\, \textrm{if } j\notin N_i \cup \{i\},\, \W = \W^\transp, \, \W \mathbf{1}_n = \mathbf{1}_n, \rho\left[\W - \frac{\mathbf{1}_n\mathbf{1}_n^\transp}{n}\right] \leq \nu < 1,
\end{equation}
where $\rho[\cdot]$ returns the spectral radius and $\nu$ is an upper bound on the value of the spectral radius. It is a common practice to generate such consensus matrices; a possible choice is the Metropolis-Hasting weighting matrix~\cite{Xiao2003,Xiao2006a}.  

A consensus iteration is a linear mapping $\mathcal{C}(\x): \x \mapsto \W \x$ with the property that the result of its repeated application converges to the mean of the initial vector, i.e., for $\x\in\mathbb{R}^n$
\begin{equation*}
\lim_{\varphi \to \infty} \underbrace{\mathcal{C}\circ \mathcal{C} \circ \cdots \circ \mathcal{C}}_{\varphi \mathrm{~times}}(\x) = \lim_{\varphi \to \infty} \W^\varphi \x =\frac{\mathbf{1}_n\mathbf{1}_n^\transp}{n} \x.
\end{equation*}
This averaging property is ensured, for example, by conditions as the ones in~\eqref{c.con}. In addition, given the structure of $\W$ in~\eqref{c.con}, each consensus iteration involves only local communications (only the neighboring nodes will share their local variables), which will be the key point of our modification. In the following, we will study multiple consensus steps, in the sense that the computing nodes will run multiple consensus iterations (each of which involving only local communications) between subsequent iterations $k$'s. We let the number of consensus steps be $\varphi \in \mathbb{N}$. In this case, the consensus mapping will be of the form $\x \mapsto \W^\varphi \x$. 
Since we will enable each node to generate its own dual variables on which consensus will be enforced, we start by defining local versions of $\mu$ and $\G$ as $\mu_i \in\mathbb{R}_+$ and $\G_i\in\mathbb{S}_+^d$, respectively. Next, we define our consensus-based dual decomposition as the following algorithm. 
\vskip0.2cm 


\noindent{\bf Consensus-based dual decomposition with primal recovery}\\  {\bf (CoBa-DD)} \hrule 
\begin{subequations}
\begin{enumerate}
\item Initialize $\mu^0_i\in\mathbb{R}_+$, $\G^0_i\in\mathbb{S}_+^d$, $i\in\myset{V}$, choose $\alpha > 0$, determine a Slater vector $\bar{\x}$ and the sets $D_{\mu}$ and $D_{\G}$ of~\eqref{itbound} with an arbitrarily picked $\tilde{\mu}, \tilde{\G}$ and a scalar $r\geq \frac{f(\bar{\x}) - q(\tilde{\mu}, \tilde{\G})}{\gamma}$; pick  a number of consensus steps $\varphi$; 
\item Local dual optimization: compute in parallel the local dual functions and their primal optimizers
\begin{equation}
q_i(\mu_i^k, \G_i^k) = \min_{x_i\in\myset{X}_i} \{L_i(x_i, \mu_i^k, \G_i^k)\}, \quad \tilde{x}_i^k = \argmin_{x_i\in\myset{X}_i} \{L_i(x_i, \mu_i^k, \G_i^k)\},
\end{equation}
as well as their subgradients $g_i(\tilde{x}_i^k)$ and $-\A_0/n-\A_i \tilde{x}_i^k$;
\item Primal recovery step: compute in parallel the ergodic sum, for $k\geq 1$
\begin{equation}
x_i^k = \frac{1}{k} \sum_{t=1}^{k} \tilde{x}_i^t;
\end{equation}
\item Update the dual variables $\mu_i^k, \G_i^k$ as
\begin{eqnarray}
\mu_i^{k+1} &=& \mathsf{P}_{D_{\mu}}\Big[\sum_{j\in\myset{V}} [\W^\varphi]_{ij} \Big( \mu_j^{k} + \alpha g_j(\tilde{x}_j^k)\Big)\Big]\\
\G_i^{k+1} &=& \mathsf{P}_{D_{\G}}\Big[\sum_{j\in\myset{V}} [\W^\varphi]_{ij} \Big(\G_j^{k} - \alpha (\A_0/n + \A_j \tilde{x}_j^k) \Big) \Big].
\end{eqnarray} 
\end{enumerate}
\label{alg}
\end{subequations}
\hrule

\vskip0.2cm
We highlight that the proposed algorithm CoBa-DD (or \eqref{alg}) involves only local communication. The only communication involved is in the $\varphi$ consensus steps, each of which requiring the nodes to share information with their neighbors. Also, note that computing $(f(\bar{\x})-q(\tilde{\mu}, \tilde{\G}))/\gamma$ (for the definition of $D_\mu$ and $D_{\G}$) is not a very difficult task, since a Slater vector is usually easy to find by inspection, and both $f(\bar{\x})$ and $\gamma$ can be computed by a consensus algorithm run in the initialization step of CoBa-DD. 


In order to analyze dual and primal convergence of~\eqref{alg}, we start by some basic results. First, given that the sets $D_{\mu}$ and $D_{\G}$ are compact, and that $\mu_i^0$ and $\G_i^0$ are picked to be bounded, the dual variables $\mu_i^k$ and $\G_i^k$ are bounded for each $k\geq 0$. In particular, we have 
\begin{equation}\label{LG}
\|\mu_i^k\|_2 \leq \Lambda < \infty, \quad \|\G_i^k\|_F \leq \Gamma < \infty.
\end{equation}


\begin{lemma}\label{l.tzeta}
Let $q(\x): \myset{X} \to \mathbb{R}$ be a concave function. Let the set $\myset{X} \subset \mathbb{R}^n$ be convex and compact, and in particular $\max_{\x \in \myset{X}} \|\x\|_2 \leq \eta$. There exist two finite scalars $\zeta > 0$ and $\tau > 0$ such that, for all $\x \in \myset{X}$, for all $g(\x)\in \partial q_{\x}(\x)$, and for all vectors $\mathbold{\nu}\in\mathbb{R}^n$ with $\|\mathbold{\nu}\|_2\leq \tau$, the following holds
$$
g(\x) + \mathbold{\nu} \in \partial_{\zeta} q_{\x}(\x).
$$
\end{lemma}
\begin{proof}
The claim is proven by using the definition of subgradient of a concave function~\eqref{egrad}. Since $q$ is a concave function, for all $\x,\y \in \myset{X}, \mathbold{\nu} \in\mathbb{R}^n$,
\begin{align*}
q(\y) -  q(\x) &\leq  \langle g(\x), \y - \x\rangle = \langle g(\x)+\mathbold{\nu},\y - \x\rangle - \langle\mathbold{\nu}, \y - \x\rangle \\
&\leq  \langle g(\x)+\mathbold{\nu}, \y - \x\rangle + \|\mathbold{\nu}\|_2\|\y - \x\|_2 \leq \langle g(\x)+\mathbold{\nu}, \y - \x\rangle + 2 \tau \eta.  
\end{align*}
For $\tau \leq \zeta/(2 \eta)$, the claim follows. \qed
\end{proof}

\begin{lemma}\label{bita}
Let the initial dual variables in \eqref{alg}, $\mu_i^0$ and $\G^0_i$ for all $i\in\myset{V}$, be bounded. Let $\W$ satisfy the conditions~\eqref{c.con}. Then, the following quantity is bounded by a certain $c_0 \geq 0$,
\begin{multline}\label{beta0}
\Big\|\sum_{j\in\myset{V}} [\W^\varphi - {\mathbf{1}_n\mathbf{1}_n^\transp}/{n}]_{ij} \Big( \mu_j^{0} + \alpha g_j(\tilde{x}_j^0)\Big)\Big\|_2 + \\  
\Big\|\sum_{j\in\myset{V}} [\W^\varphi - {\mathbf{1}_n\mathbf{1}_n^\transp}/{n}]_{ij} \Big(\G_j^{0} - \alpha (\A_0/n + \A_j \tilde{x}_j^0) \Big) \Big\|_{\mathrm{F}}\leq c_0, \quad \forall i\in\myset{V}.
\end{multline}
\end{lemma}
\begin{proof}
The proof follows given the compactness of $X$ and (therefore) the boundedness of the subgradients.  \qed
\end{proof}

We now present the main convergence results. 

\begin{theorem}\label{th.1}
\emph{(Dual variable agreement)} Let $\bar{\mu}^k, \bar{\G}^k$ be the mean values of the dual variables generated via the algorithm~\eqref{alg}, i.e., 
$$
\bar{\mu}^k = \frac{1}{n}\sum_{i\in\myset{V}}{\mu_i^k}, \quad \bar{\G}^k = \frac{1}{n}\sum_{i\in\myset{V}}{\G_i^k}.
$$
Let Assumptions~\ref{as.0} till \ref{as.slater} hold and let $\W$ satisfy the conditions~\eqref{c.con}. Let $\mu_i^0$ and $\G^0_i$ for $i\in\myset{V}$ be bounded and let $\beta_0\geq c_0$, with $c_0$ defined as in~\eqref{beta0}. Define $L$ and $Q$ as in~\eqref{LQ} and let 
$$
M:= L+Q,\, \quad p : = \frac{\nu^\delta\beta_0}{\beta_0 + \alpha M}.
$$
There exists a number of consensus iterations $\bar{\varphi}$, such that if $\varphi\geq \bar{\varphi} + \delta$, $\delta \geq 0$, $k \geq 1$, then the dual variables reach consensus as
\begin{align*}
\|\mu_i^{k+1} - \bar{\mu}^{k+1}\|_2 &\leq 2 p^{k-1} \nu^\delta \beta_0 + 2 p\, \alpha M \frac{1-p^{k-1}}{1-p}, \quad\forall i\in\myset{V},\\
\|\G^{k+1}_i - \bar{\G}^{k+1}\|_{\mathrm{F}} &\leq 2 p^{k-1} \nu^\delta \beta_0 + 2 p\, \alpha M \frac{1-p^{k-1}}{1-p}, \quad\forall i\in\myset{V}.
\end{align*}
Furthermore, 
\begin{equation*}
\bar{\varphi} = \frac{\log(\beta_0) - \log(4 n (1+d^2) (\beta_0 + \alpha M))}{\log (\nu)}. 
\end{equation*}
\end{theorem}

\begin{corollary}\label{co.1}
Under the same conditions of Theorem~\ref{th.1}, we obtain
\begin{equation*}
\lim_{k\to\infty} \|\mu_i^k - \bar{\mu}^k\|_2 \leq \frac{ 2 p \,\alpha M}{1-p},\quad  \lim_{k\to\infty} \|\G_i^k - \bar{\G}^k\|_{\mathrm{F}} \leq \frac{2 p\, \alpha M}{1-p}, \, \quad\forall i\in\myset{V}. 
\end{equation*}
\end{corollary}

Theorem~\ref{th.1} and Corollary~\ref{co.1} specify how the consensus is reached among the nodes on the value of the dual variables while the algorithm~\eqref{alg} is running. Specifically, the consensus is reached exponentially fast to a steady-state bounded error floor. This bounded error depends on $\alpha$ (which can be tuned), and on $p$, which can also be tuned by varying $\varphi$. In particular, for $\varphi \to \infty$, due to the fact that $\nu < 1$ in conditions~\eqref{c.con}, then $p = 0$ and we obtain back the usual dual decomposition scheme with perfect agreement among the nodes.

\begin{remark}\label{remphi}
Computing the lower bound on the number of consensus steps $\bar{\varphi}$ can be done during the initialization of the algorithm. We can always pick $\beta_0$ big enough so that $\beta_0 \gg \alpha M$, which means that $\bar{\varphi}$ can be simplified as 
$
\bar{\varphi} = \frac{\log(1/(4 n (1+d^2)))}{\log (\nu)}, 
$
which can be determined in a distributed way~\cite{Kempe2008}. 
\end{remark}


\begin{theorem}\label{th.2} \emph{(Dual objective convergence)} Let ${\mu}^k, {\G}^k$ be the dual variables generated via the algorithm~\eqref{alg}.  Let $\mu_i^0$ and $\G^0_i$ for all $i\in\myset{V}$ be bounded and let $\beta_0$ be defined as in~Theorem~\ref{th.1}. Define $L$ and $Q$ as in~\eqref{LQ} and let $M:= L+Q$. Choose a scalar $\tau$ such that $\beta_0 / \alpha \leq \tau$. Let $\zeta$ be defined as in Lemma~\ref{l.tzeta} for the concave function $q(\mu,\G)$ and the choice of $\tau$. Let $q^*$ be the optimal value of $q(\mu,\G)$. Let Assumptions~\ref{as.0} till \ref{as.slater} hold and let $\W$ satisfy the conditions~\eqref{c.con}. Let $\varphi \geq \bar{\varphi} + \delta, \delta \geq 0$ and let $\bar{\varphi}$ be defined as in Theorem~\ref{th.1}. The following holds true.

\noindent If $q^* = \infty$, then 
$$
\limsup_{k\to \infty} q(\mu^k_i, \G_i^k) = \infty, \quad \forall i\in\myset{V},
$$

\noindent If $q^* < \infty$, then 
$$
\limsup_{k\to \infty} q(\mu^k_i, \G_i^k) \geq q^* - \alpha n (M + \tau)^2/2 - n(\beta_\infty(9 M + 3\tau) + \zeta), \quad \forall  i\in\myset{V},
$$
with $\beta_{\infty} = \frac{p\,\alpha M}{1-p}$ and $p = \frac{\nu^\delta \beta_0}{\beta_0 + \alpha M}$. 

\end{theorem}

Theorem~\ref{th.2} implies dual objective convergence up to a bounded error floor. Convergence is even more evident if we remember that, owning to optimality, $q(\mu^k_i, \G_i^k) \leq q^*$, and thus, if we define $q_i^{\infty} : = \limsup_{k\to \infty} q(\mu^k_i, \G_i^k)$, we obtain
$$
0 \geq q_i^{\infty} - q^* \geq - \alpha n (M + \tau)^2/2 - n(\beta_\infty(9 M + 3\tau) + \zeta) =: -\varepsilon^2.
$$
Note that the rightmost term ($-\varepsilon^2$) represents a measure of sub-optimality of the approximate solution. 
%


\begin{theorem}\label{th.3}
\emph{(Primal objective convergence)} Let ${\mu}^k, {\G}^k, \x^k$ be the dual and primal variables generated via the algotithm~\eqref{alg}.  Let $\mu_i^0$ and $\G^0_i$ for all $i\in\myset{V}$ be bounded and let $\beta_0$ be defined as in~Theorem~\ref{th.1}. Define $L$ and $Q$ as in~\eqref{LQ}, $\Lambda$ and $\Gamma$ as in~\eqref{LG}, and let $M:= L+Q$. Choose a scalar $\tau$ such that $\beta_0 / \alpha \leq \tau$. Let $\zeta$ be defined as in Lemma~\ref{l.tzeta} for the concave function $q(\mu,\G)$ and the choice of $\tau$. Let $f^*$ be the optimal value of $f(\x)$. Let Assumptions~\ref{as.0} till \ref{as.slater} hold and let $\W$ satisfy the conditions~\eqref{c.con}. Let $\varphi \geq \bar{\varphi} + \delta, \delta \geq 0$ and let $\bar{\varphi}$ be defined as in Theorem~\ref{th.1}. The following holds true.

\begin{enumerate}
\item[(a)] An upper bound on the primal cost of the vector $\x^k$, $k\geq 1$, is given by
\begin{equation*}
f(\x^k) \leq f^* + \frac{\Lambda^2 + \Gamma^2}{2 k \alpha/n} +e_k;
\end{equation*}
\item[(b)] A lower bound on the primal cost of the vector $\x^k$, $k\geq 1$, is given by
\begin{equation*}
f(\x^k) \geq f^* - \frac{9(\Lambda^2 + \Gamma^2)}{2 k \alpha/n} - e_k;
\end{equation*}
\end{enumerate}
where 
$$
e_k =   \frac{\alpha n(M+\tau)^2}{2} + n \tau (\Lambda + \Gamma) +  n (\beta_0(6M + 3\tau) + \zeta).
$$
\end{theorem}

Theorem~\ref{th.3} formulates convergence of the primal cost up to an error bound $e_k$. The rate of convergence is $O(1/k)$. We can also distinguish the error terms that come from the constant stepsize $\alpha$ and the terms that come from the finite number of consensus steps $\varphi$. In particular, we can write
$$
e_k =  \underbrace{\frac{\alpha n M^2}{2}}_{(1)} + \underbrace{\frac{\alpha n (2M\tau + \tau^2)}{2} + n \tau (\Lambda + \Gamma) +  n (\beta_0(6M + 3\tau) + \zeta)}_{(2)},
$$
and see that the term (1) is due to the constant stepsize, while the term (2) is due to the finite number of consensus steps. Furthermore, if $\varphi \to \infty$, then $c_0 = 0$, and we can set $\beta_0 = \tau = \zeta = 0$, yielding
$$
\lim_{\varphi\to\infty} e_k = \frac{\alpha n M^2}{2}.
$$
This is similar to the error level we obtain for the dual decomposition method in~\eqref{it.dds}, and Theorem~\ref{th.0}. 
Theorem~\ref{th.3} defines the main trade-offs in designing the algorithm's parameters $\alpha$ and $\varphi$. The larger the stepsize $\alpha$ is, the faster the convergence is, even though the steady-state error becomes larger. If we increase $\varphi$ then the communication effort increases and the error $e_k$ decreases. 

\section{Proof of Theorem~\ref{th.1} and Theorem~\ref{th.2}}\label{sec:6}

\subsection{Preliminaries}

We start our analysis by rewriting Step 4 of \eqref{alg} in a more compact way. Let $\z_i \in \mathbb{R}^{1+d^2}$ be the vector defined as  $\z_i := (\mu_i, \vec(\G_i)^\transp)^\transp$, and let $\zv$ be the stacked vector of all the $\z_i$, $i\in\myset{V}$. Similarly, let $\h_i(\x)$ be the vector $\h_i(\x):= (g_i(x_i), \vec(-\A_0/n - \A_i x_i)^\transp)^\transp$, and let $\hv(\x)$ the stacked vector of all the $\h_i(\x)$, $i\in\myset{V}$. Let $Z$ be the convex set 
\begin{equation}\label{set}
Z:= \{\z := (\mu, \vec(\G)^\transp)^\transp \in \mathbb{R}^{1+d^2}| \mu \in D_{\mu}, \G \in D_{\G} \},
\end{equation}
and let $\Zv = \prod_{i=1}^n Z$. The iterations in Step 4 of \eqref{alg} can be rewritten as 
\begin{equation}\label{it}
\zv^{k+1} = \mathsf{P}_{\Zv}\left[ \W^\varphi \otimes {\bf I}_{{1+d^2}} \left(\zv^k + \alpha \hv(\tilde{\x}^k)\right) \right]. 
\end{equation}
The iteration~\eqref{it} represents a consensus-based subgradient method to maximize the dual function $q(\mu, \G)$, i.e, the maximization problem 
\begin{equation*}
q^* : = \max_{\mu \in D_{\mu}, \G \in D_{\G}} \, \sum_{i\in\myset{V}} q_i(\mu, \G) \equiv  \max_{\z \in Z} \, \sum_{i\in\myset{V}} q_i(\z), \quad \textrm{for } \z =  (\mu, \vec(\G)^\transp)^\transp.
\end{equation*}
In particular~\eqref{it} assigns to each node a copy of $\z$, $\z_i$, and enforces consensus among them. Furthermore, by~\eqref{LQ}, by triangle inequality, and by~\eqref{LG},
\begin{subequations}
\begin{align}\label{MM}
\|\h_i(\x)\|_2 &\leq \|h_i(\x)\|_2 + \|\Q_i(\x)\|_{\mathrm{F}} = L + Q = M, \quad \|\hv(\x)\|_2 \leq nM,\\ \max_{\z \in Z}\|\z\|_2 &\leq \sqrt{\Lambda^2 + \Gamma^2} \leq \Lambda + \Gamma.
\end{align}
\end{subequations}

\begin{lemma}\label{l.basic}\emph{(\cite[Lemma~1]{Johansson2008})}
Let $\x_i\in\mathbb{R}^m, i\in\myset{V}$ be $m$-dimensional vectors. Let $\bar{\x}$ be the average value of $\x_i, i\in\myset{V}$, i.e., $\bar{\x} = \frac{1}{n}\sum_{i\in\myset{V}}\x_i$. The following basic relations hold, 
\begin{enumerate}
\item[(a)] if $\|\x_i - \x_j\|_2 \leq \beta,\, \forall i,j\in\myset{V}$, then $\|\x_i - \bar{\x}\|_2 \leq \frac{n-1}{n} \beta$;
\item[(b)] if $\|\x_i - \bar{\x}\|_2 \leq \beta, \, \forall i\in\myset{V}$, then $\|\x_i - \x_j\|_2 \leq 2\beta$.
\end{enumerate}
\end{lemma}

\begin{lemma}\label{l.basic2}\emph{(\cite[Lemma~2]{Johansson2008})}
Let $\x^k \in\mathbb{R}^n$ be an $n$-dimensional vector, with components $x_i\in\mathbb{R}, i=1,\dots,n$. Let $\x^{k+1} = \W^\varphi \x^k$, with $\W\in\mathbb{R}^{n\times n}$ fulfilling conditions~\eqref{c.con}. Let $\|x^k_i - x^k_j\|_2 \leq \sigma$, for a bounded $\sigma$, and for all $i,j = 1,\dots,n$. Then  $\|x^{k+1}_i - x^{k+1}_j\|_2 \leq 2 \nu^\varphi n \sigma$ for all $i,j = 1,\dots,n$.
\end{lemma}

\begin{lemma}\label{agr}
Let $\{\zv^k\}$ be generated by~\eqref{it} under Assumptions~\ref{as.0} till \ref{as.slater}. Let $\v^k_i\in\mathbb{R}^{1+d^2}$, for all $i\in\myset{V}$ be defined as 
\begin{equation*}
\v^{k}_i = \sum_{j\in V}[\W^\varphi]_{ij} \left(\z_j^k + \alpha \h_j(\tilde{\x}^k)\right), 
\end{equation*}
and let $\bar{\v}^k$ be the average value of $\v^k_i, i\in\myset{V}$, i.e., $\bar{\v}^k = \frac{1}{n}\sum_{i\in\myset{V}}\v_i^k$. There exists a $\bar{\varphi}\geq 1$, such that if $\varphi\geq \bar{\varphi}+\delta$ with $\delta\geq 0$, then
$$ \|\v^k_i - \bar{\v}^k\|_2\leq \beta, \, \forall i \in \myset{V}\implies \|\v^{k+1}_i - \bar{\v}^{k+1}\|_2\leq \nu^\delta \beta, \, \forall i \in \myset{V}, k\geq 0. $$
\end{lemma}

\begin{proof}
The proof is an adaptation of \cite[Lemma~3]{Johansson2008}. In particular, we can show that for all $i,j\in V$
\begin{equation}\label{dummy10} \|\v^k_i - \bar{\v}^k\|_2\leq \beta \implies  \|\v^{k+1}_i - {\v}_j^{k+1}\|_2\leq 4\nu^{\varphi} n (1+d^2)( \beta + \alpha M). \end{equation}
Therefore, if we choose, 
\begin{equation*}
\varphi \geq \underbrace{\frac{\log(\beta) - \log(4 n (1+d^2) (\beta + \alpha M))}{\log (\nu)}}_{=: \bar{\varphi}} + \delta, \quad \delta \geq 0,
\end{equation*}
\begin{equation*} \hskip-0.56cm\textrm{then,\quad} \hskip0.56cm \|\v^k_i - \bar{\v}^k\|_2\leq \beta,\,\, \forall i \in \myset{V}\implies \|\v^{k+1}_i - {\v}_j^{k+1}\|_2\leq \nu^{\delta} \beta,\,\, \forall i,j \in \myset{V}, \end{equation*}
and the claim follows from Lemma~\ref{l.basic}.\emph{(a)}. In order to prove~\eqref{dummy10}, we proceed as follows. 
\begin{multline*}
\|\v^k_i - \bar{\v}^k\|_2\leq \beta,  \, \forall i \in \myset{V} \underbrace{\implies}_{\textrm{Lemma~\ref{l.basic}}} \|\v^k_i - {\v}^k_j\|_2\leq 2 \beta,\, \forall i,j \in \myset{V}   \\
{\implies} \|[\v^k_i - {\v}^k_j]_\ell\|_2\leq 2 \beta,\, \forall i,j \in \myset{V}, \,\ell=1,\dots,1+d^2, 
\end{multline*}
where $[\cdot]_\ell$ extracts the $\ell$-th component of a vector. Define 
$$
\uu_i^{k+1} = \mathsf{P}_{Z}[\v^k_i] + \alpha \h_i(\tilde{\x}^{k+1}), \quad \forall i\in\myset{V}.  
$$
Prior to consensus, the distance between the iterates can be bounded as
\begin{multline*}
\|\uu_i^{k+1}  - \uu_j^{k+1} \|_2 = \|\mathsf{P}_{Z}[\v^k_i] + \alpha \h_i(\tilde{\x}^{k+1}) - \mathsf{P}_{Z}[\v^k_j] - \alpha \h_j(\tilde{\x}^{k+1}) \|_2 \\ \leq 
\|\mathsf{P}_{Z}[\v^k_i] - \mathsf{P}_{Z}[\v^k_j] \|_2 + 2\alpha M   \leq \|\v^k_i- \v^k_j \|_2 +  2\alpha M \leq 2(\beta+\alpha M), 
\end{multline*}
which also implies $\|[\uu_{i}^k - \uu_j^k]_\ell\|_2\leq 2(\beta+\alpha M)$. Given that $\z_i^{k+1} = \mathsf{P}[\v_i^k], \forall i$, after consensus, we have
\begin{multline}\label{dummy11}
\|\v^{k+1}_i - \v^{k+1}_j \|_2 = \Big\|\sum_{p\in\myset{V}} [\W^\varphi]_{ip} \uu^{k+1}_{p} -  \sum_{p\in\myset{V}} [\W^\varphi]_{jp} \uu^{k+1}_{p}\Big\|_2 \\ \leq \sum_{\ell=1}^{1+d^2} \Big\|\Big[\sum_{p\in\myset{V}} [\W^\varphi]_{ip} \uu^{k+1}_{p} -  \sum_{p\in\myset{V}} [\W^\varphi]_{jp} \uu^{k+1}_{p} \Big]_\ell \Big\|_2 \\ = \sum_{\ell=1}^{1+d^2} \Big\|\sum_{p\in\myset{V}} [\W^\varphi]_{ip} [\uu^{k+1}_{p}]_\ell -  \sum_{p\in\myset{V}} [\W^\varphi]_{jp} [\uu^{k+1}_{p}]_\ell \Big\|_2 \\ = \sum_{\ell=1}^{1+d^2} \Big\|[\W^\varphi \tilde{\uu}_\ell^{k+1}]_i -  [\W^\varphi \tilde{\uu}_\ell^{k+1}]_j \Big\|_2, 
\end{multline}
where $\tilde{\uu}_\ell^{k+1} = ([\uu^{k+1}_{1}]_\ell, \dots,[\uu^{k+1}_{n}]_\ell)^\transp$. As said $\|[\uu_{i}^k - \uu_j^k]_\ell\|_2\leq 2(\beta+\alpha M)$ which means $\|[\tilde{\uu}_\ell^k]_i - [\tilde{\uu}_\ell^k]_j\|_2\leq 2(\beta+\alpha M)$. Thus, by using Lemma~\ref{l.basic2} we can bound~\eqref{dummy11} as
\begin{equation*}
\|\v^{k+1}_i - \v^{k+1}_j \|_2 \leq \sum_{\ell=1}^{1+d^2} \Big\|[\W^\varphi \tilde{\uu}_\ell^{k+1}]_i -  [\W^\varphi \tilde{\uu}_\ell^{k+1}]_j \Big\|_2 \leq 4 \nu^\varphi n (1+d^2)( \beta + \alpha M),
\end{equation*} 
which is the rightmost term in~\eqref{dummy10} and the claim is proven.
\qed
\end{proof}

\subsection{Proof of Theorem~\ref{th.1}}

The quantity $\|\v^{0}_i - \bar{\v}^{0}\|_2$ is upper bounded by $\beta_0\geq c_0$ by Lemma~\ref{bita} (inequality~\eqref{beta0}), thus, $\|\v^{0}_i - \bar{\v}^{0}\|_2\leq \beta_0$. Let us choose $\varphi\geq \bar{\varphi} + \delta$, $\delta\geq 0$, with $\bar{\varphi}$ determined as in Theorem~\ref{th.1}. Then, by Lemma~\ref{agr} and~\eqref{dummy10}, it follows that, 
\begin{eqnarray*}
\|\v^{1}_i - \bar{\v}^{1}\|_2 &\leq& \nu^\delta \beta_0\\
\|\v^{2}_i - \bar{\v}^{2}\|_2 &\leq& 4\nu^{\varphi} n (1+d^2)( \nu^\delta \beta_0 + \alpha M) = \nu^\delta \beta_0 \frac{\nu^\delta \beta_0 + \alpha M}{\beta_0 + \alpha M} \\
\|\v^{3}_i - \bar{\v}^{3}\|_2 &\leq&  \nu^\delta \beta_0 \frac{\nu^\delta \beta_0}{\beta_0 + \alpha M}\,\Big( \frac{\nu^\delta \beta_0 + \alpha M}{\beta_0 + \alpha M} + \alpha M\Big) \\
\|\v^{k}_i - \bar{\v}^{k}\|_2 &\leq& \nu^\delta \beta_0 \left(\frac{\nu^\delta \beta_0} {\beta_0 + \alpha M}\right)^{k-1} +  \alpha M \left(-1 + \sum_{t=0}^{k-1} \left(\frac{\nu^\delta \beta_0 }{\beta_0 + \alpha M}\right)^{t}\right).
\end{eqnarray*}
Let $p:= \frac{\nu^\delta \beta_0 }{\beta_0 + \alpha M}$, since $p < 1$, then 
\begin{equation}\label{dummy3}
\|\v^{k}_i - \bar{\v}^{k}\|_2 \leq p^{k-1} \nu^\delta \beta_0 + p \alpha M  \frac{1-p^{k-1}}{1-p}=:\beta_k,\quad k\geq 1
\end{equation}
and by Lemma~\ref{l.basic}.\emph{(b)}, we derive $\|\v^{k}_i - {\v}^{k}_j\|_2 \leq 2 \beta_k$. 
By using the non-expansive property of the projection operator, since $\z_i^{k+1} = \mathsf{P}[\v_i^k]$, for all $i$, we can write 
\begin{equation}\label{dummy1}
\|{\z}_{i}^{k+1} -{\z}_{j}^{k+1}\|_2 \leq \|\v^{k}_i - {\v}^{k}_j\|_2 \leq 2\beta_{k}, \quad k\geq 1,
\end{equation}
and by Lemma~\ref{l.basic}.\emph{(a)} the claim follows. 
\qed

\subsection{Proof of Theorem~\ref{th.2}}

%
We define an average value for $\zv^k$ as $\bar{\z}^{k} = \frac{1}{n}\sum_{i\in\myset{V}} \z_i^k$. For convergence purposes, we need to keep track of the difference $\bar{\z}^{k+1} - \mathsf{P}_{Z}[\bar{\v}^k]$, and thus we define the vectors $\y^k\in\mathbb{R}^{1+d^2}$ and $\d^k\in\mathbb{R}^{1+d^2}$ as 
\begin{equation}\label{ydef}
\y^k := \mathsf{P}_{Z} [\bar{\v}^{k-1}], \quad \d^k := \bar{\z}^k - \y^k, \quad k\geq 1.
\end{equation}
The main idea of the proof is to show that $\y$ is updated via an approximate $\epsilon$-subgradient method and, then, by using~\cite[Proposition~4.1]{Nedich2002} the theorem follows. The first part is formalized in the following lemma. 

\begin{lemma}\label{subgrad}
Let $\y^k$ be defined as in~\eqref{ydef}. Under the same conditions of Theorem~\ref{th.2}, for all $k\geq 1$,
\begin{enumerate}
\item[(a)] The quantity $\|{\d^k}/{\alpha}\|_2$ is upper bounded by $\beta_{k-1}/\alpha \leq \tau$ (where $\beta_k$ is defined in~\eqref{dummy3});
\item[(b)] The following inequalities are true,  for all $i\in V$
\begin{align}\label{dummy2}
q(\y^k) &\leq q(\z_i^k) + 3 n M \beta_{k-1} \\ 
\label{dummy5}
q_i(\y) & \leq q_i(\y^k) + \langle\h_i(\tilde{\x}^k)+\mathbold{\nu}, \y - \y^k\rangle + \epsilon_{k}/n, \quad \forall \y \in Z.
\end{align}

\item[(c)] The quantity $\g(\tilde{\x}^k) := \sum_{i\in\myset{V}}\left( \h_i(\tilde{\x}^k) + \frac{\d^k}{\alpha} \right)$ is an $\epsilon_k$-subgradient of $q (\y^k)$ with respect to $\y$.
\item[(d)] The variable $\y^k$ is updated via an $\epsilon$-subgradient method
\begin{equation}\label{iteps}
\y^{k+1} = \mathsf{P}_{Z} \Big[\y^k + \frac{\alpha}{n} \g(\tilde{\x}^k) \Big], \quad \g(\tilde{\x}^k) \in \partial_{\epsilon_k}q_{\y} (\y^k). 
\end{equation}
\end{enumerate}
And $\epsilon_k = n (\beta_{k-1}(6 M + 3\tau) +\zeta)$.
\end{lemma}
\begin{proof}\color{white}{a}\color{black}

\noindent\emph{(a)} We start by bounding $\|\d^k\|_2$, 
\begin{equation*}
\|\d^k\|_2 = \frac{1}{n} \Big\|\sum_{i\in\myset{V}} \Big(\mathsf{P}_Z[\v_i^{k-1}] - \mathsf{P}_Z[\bar{\v}^{k-1}] \Big)\Big\|_2 \leq \frac{1}{n} \sum_{i\in\myset{V}} \| \v_i^{k-1}-\bar{\v}^{k-1} \|_2 \leq \beta_{k-1},
\end{equation*}
where we have used the inequality~\eqref{dummy3} to bound the term $\| \v_i^{k-1}-\bar{\v}^{k-1} \|_2$. 

\noindent \emph{(b)} Since $\y^k\in Z$ and $\z_i^k \in Z$, by the concavity of $q_i(\z)$ and the definition of subgradient of a concave function~\eqref{egrad}, we can write for all $i,j \in V$
\begin{align*}
q_j(\y^k) &\leq q_j(\z_i^k) + \langle\h,\y^k-\z_i^k\rangle,\quad \textrm{where } \h \in \partial q_{j,\z}(\z_i^k)\\ 
&\leq q_j(\z_i^k) + \|\h\|_2 \|\z_i^k - \y^k\|_2 \leq q_j(\z_i^k) + M ( \|\z_i^k - \bar{\z}^k\|_2 + \|\d^k\|_2) \\
& \leq q_j(\z_i^k) + M (2 \beta_{k-1} + \beta_{k-1}) \leq q_j(\y^k) + 3 M \beta_{k-1}.
\end{align*}
In particular, we have used the fact that any subgradient vector of $q_j(\z)$ is bounded by $M$ \eqref{MM}, and inequality~\eqref{dummy1}. If we sum the last relation over $j\in\myset{V}$, we obtain~\eqref{dummy2}.
In addition for any $\y\in Z$, by using Lemma~\ref{l.tzeta}
\begin{align*}
q_i(\y) &\leq q_i(\z_i^k) + \langle\h_i(\tilde{\x}^k), \y\hskip-0.05cm - \hskip-0.05cm\z_i^k\rangle\leq q_i(\z_i^k) + \langle\h_i(\tilde{\x}^k)+\mathbold{\nu},\y \hskip-0.05cm-\hskip-0.05cm \z_i^k\rangle+\zeta\\ 
&\leq q_i(\y^k) + \langle\h_i(\tilde{\x}^k)+\mathbold{\nu},\y\hskip-0.05cm -\hskip-0.05cm \z_i^k\rangle + 3 M \beta_{k-1} + \zeta  \\ 
&= q_i(\y^k) + \langle\h_i(\tilde{\x}^k)+\mathbold{\nu},\y\hskip-0.05cm -\hskip-0.05cm \y^k + \y^k\hskip-0.05cm - \hskip-0.05cm\z_i^k\rangle \hskip-0.05cm+\hskip-0.05cm 3 M \beta_{k-1} \hskip-0.05cm+\hskip-0.05cm \zeta \\
&\leq q_i(\y^k) + \langle\h_i(\tilde{\x}^k)+\mathbold{\nu},\y \hskip-0.05cm- \hskip-0.05cm\y^k\rangle \hskip-0.05cm+\hskip-0.05cm \|\h_i(\tilde{\x}^k)\hskip-0.05cm+\hskip-0.05cm\mathbold{\nu}\|_2 \|\y^k \hskip-0.05cm-\hskip-0.05cm \z_i^k\|_2 \hskip-0.05cm+\hskip-0.05cm 3 M \beta_{k-1} \hskip-0.05cm+\hskip-0.05cm \zeta.
\end{align*}
We use the fact that $\|\mathbold{\nu}\|_2\leq \tau$ by construction in Lemma~\ref{l.tzeta}, $\|\h_i(\tilde{\x}^k)\|_2 \leq M$ by~\eqref{MM}, $\|\z_i^k - \bar{\z}^k\|_2 \leq 2\beta_{k-1}$ by~\eqref{dummy1}, and $\|\d^k\|_2\leq \beta_{k-1}$ by the preceding proof. By using these inequalities, we can bound
$$
\|\h_i(\tilde{\x}^k)+\mathbold{\nu}\|_2 \leq M + \tau, \quad \|\y^k - \z_i^k\|_2 = \|\z_i^k - \bar{\z}^k + \d^k\|_2 \leq 3 \beta_{k-1},
$$
and we obtain
\begin{equation*}
q_i(\y) \leq q_i(\y^k) + \langle\h_i(\tilde{\x})+\mathbold{\nu},\y - \y^k\rangle + (\beta_{k-1}(6M +3\tau) + \zeta), 
\end{equation*}
which is~\eqref{dummy5}. 

\noindent \emph{(c)} By using the definition of subdifferential~\eqref{egrad}, the inequality~\eqref{dummy5} implies $(\h_i(\tilde{\x})+\mathbold{\nu}) \in \partial_{\epsilon_k/n} q_{i,\y}(\y)$ with $\epsilon_k/n = (\beta_{k-1}(6M +3\tau) + \zeta)$. Summation over $i$ yields, 
\begin{equation*}
q(\y) \leq q(\y^k) + \Big\langle\sum_{i\in\myset{V}}\h_i(\tilde{\x})+\mathbold{\nu}, \y - \y^k\Big\rangle + n(\beta_{k-1}(6M +3\tau) + \zeta), 
\end{equation*}
for any $\mathbold{\nu}$, such that $\|\mathbold{\nu}\|\leq \tau$. Since $\|\d^k/\alpha\|_2 \leq \tau$ by construction, then we can choose $\mathbold{\nu} = \d^k/\alpha$, from which the claim follows. 

\noindent \emph{(d)} It is sufficient to write explicitly the update rule for $\y^k$. Starting from the definition of $\y^{k+1}$ in~\eqref{ydef} and the definition of $\v_i^k$ in~Lemma~\ref{agr}, we obtain
\begin{eqnarray*}
\y^{k+1} &=& \mathsf{P}_{Z} \Big[\frac{1}{n}\sum_{i\in\myset{V}} \v_i^k \Big] = \mathsf{P}_{Z} \Big[\frac{1}{n}\sum_{i\in\myset{V}}\sum_{j\in\myset{V}}[\W^\varphi]_{ij}(\z_j^k + \alpha \h_j(\tilde{\x}^k)) \Big] \\ &=& \mathsf{P}_{Z} \Big[\frac{1}{n}\sum_{i\in\myset{V}} (\z_i^k + \alpha \h_i(\tilde{\x}^k)) \Big] =  \mathsf{P}_{Z} \Big[\y^k + \d^k + \frac{\alpha}{n}\sum_{i\in\myset{V}} \h_i(\tilde{\x}^k)) \Big] \\ &=& \mathsf{P}_{Z} \Big[\y^k + \frac{\alpha}{n}\left(\sum_{i\in\myset{V}} \h_i(\tilde{\x}^k) + \frac{\d^k}{\alpha}\right) \Big].
\end{eqnarray*}
Given part \emph{(c)} of this Lemma, the claim follows. 
\qed
\end{proof}


\begin{proof} (of Theorem~\ref{th.2}) By Lemma~\ref{subgrad}, the sequence $\{\y^k\}$ is generated via an $\epsilon_k$ subgradient algorithm to maximize $q(\y)$. And in particular, $k\geq 1$
$$
\y^{k+1} = \mathsf{P}_Z[\y^k + \alpha/n \g(\tilde{\x}^k)], \quad \|\g(\tilde{\x}^k)\|_2 \leq n (M+\tau).
$$
Therefore, we can use any standard result on the convergence of approximate subgradient algorithms. E.g., by using~\cite[Proposition~4.1]{Nedich2002} (with $m=1$), the following holds for the sequence $\{\y^k\}$,

\noindent If $q^* = \infty$, then 
$$
\limsup_{k\to \infty} q(\y^k) = \infty,
$$

\noindent If $q^* < \infty$, then 
$$
\limsup_{k\to \infty} q(\y^k) \geq q^* - \alpha n (M + \tau)^2/2 - n(\beta_\infty(6 M + 3\tau) + \zeta),
$$
where $\beta_\infty = \lim_{k\to\infty} \beta_{k-1}$. Then, from the inequality~\eqref{dummy2} the claim is proven. 
\qed 
\end{proof}

\section{Primal Recovery: Proof of Theorem~\ref{th.3}}\label{sec:7}


\subsection{Some Basic Facts}


\begin{lemma}\label{itiy}
Let $\y^k$ be defined as~\eqref{ydef}. Under the same assumptions and notation of Theorem~\ref{th.2}, 
\begin{enumerate}
\item[(a)] For any $\y\in Z$,
\begin{equation*}
\sum_{t=1}^k \langle\g(\tilde{\x}^t),\y-\y^{t}\rangle \leq \frac{\|\y^{1} -  \y\|^2_2}{2 \alpha/n} + k\,\frac{\alpha n(M+\tau)^2}{2};
\end{equation*}
\end{enumerate}
\begin{enumerate}
\item[(b)] For any $\y\in Z$,
\begin{equation*}
\sum_{t=1}^k \langle\g(\tilde{\x}^t),\y-\y^{*}\rangle \leq \frac{\|\y^{1} -  \y\|^2_2}{2 \alpha/n} + k\,\frac{\alpha n(M+\tau)^2}{2} + \sum_{t=1}^k \epsilon_t,
\end{equation*}
where ${\epsilon_{t}} = n (\beta_{t-1} (6M + 3\tau)+\zeta)$.
\end{enumerate}
\end{lemma}

\begin{proof}
We start from the update rule~\eqref{iteps}. For any $\y \in Z$,
\begin{eqnarray*}
\|\y^{k+1} - \y\|^2_2 &=& \Big\|\mathsf{P}_{Z}\Big[\y^{k} + \frac{\alpha}{n} \g(\tilde{\x}^k)\Big] - \mathsf{P}_{Z}[ \y] \Big\|^2_2 \leq  \Big\|\y^{k} + \frac{\alpha}{n} \g(\tilde{\x}^k) -  \y\Big\|^2_2\\
&\leq & \|\y^{k} -  \y\|^2_2 + \frac{2 \alpha}{n} \langle\g(\tilde{\x}^k),\y^{k} -  \y\rangle + {\alpha^2} (M+\tau)^2.
\end{eqnarray*}
where we use the fact that $\|\g(\tilde{\x}^k)\|_2 = \|\sum_{i\in\myset{V}}(\h_i(\tilde{\x}^k)+\d^k/\alpha) \|_2 \leq n (M+\tau)$. Therefore, for any $\y \in Z$ 
\begin{equation}\label{dummy4}
\langle\g(\tilde{\x}^k),\y-\y^{k}\rangle \leq \frac{\|\y^{k} -  \y\|^2_2- \|\y^{k+1} -  \y\|^2_2}{2 \alpha/n} + \frac{\alpha\, n(M+\tau)^2}{2},
\end{equation}
and by summing over $k$, part \emph{(a)} follows. Since $\g(\tilde{\x}^k)$ is an $\epsilon_k$-subgradient of the dual function $q$ at $\y^k$, using the subgradient inequality~\eqref{egrad}, 
$$
\langle\g(\tilde{\x}^k),\y^{k} -  \y^*\rangle \leq q(\y^k) - q(\y^*)+\epsilon_k \leq \epsilon_k,  
$$
where the last inequality comes from the optimality condition  $q(\y^k) \leq q(\y^*)$, which is valid for any $\y^k\in\myset{Z}$. In particular, $\epsilon_k$ is defined in Lemma~\ref{subgrad}.\emph{(c)}. We then have, 
$$
\langle\g(\tilde{\x}^k),\y -  \y^*\rangle = \langle\g(\tilde{\x}^k),\y -  \y^{k}\rangle + \langle\g(\tilde{\x}^k),\y^{k} -  \y^*\rangle \leq  \langle\g(\tilde{\x}^k),\y -  \y^{k}\rangle +\epsilon_k. 
$$
From the preceding relation and~\eqref{dummy4}, we obtain
\begin{equation*}
\langle\g(\tilde{\x}^k),\y-\y^{*}\rangle \leq \frac{\|\y^{k} -  \y\|^2_2- \|\y^{k+1} -  \y\|^2_2}{2 \alpha/n} + \frac{\alpha n(M+\tau)^2}{2} + \epsilon_k, \quad k\geq 1
\end{equation*}
and summing over $k$ part \emph{(b)} follows as well. In particular, we remark that $\y^1 = \mathsf{P}_Z[\bar{\v}^0]$, which is bounded, since $Z$ is a compact set. 
\qed
\end{proof}






\subsection{Proof of Theorem~\ref{th.3}.\emph{(a)}}
\begin{proof}
By convexity of the primal cost $f(\x)$ and the definition of $\tilde{x}^k_i$ as a minimizer of the local Lagrangian functions over $x_i\in\myset{X}_i$, we have, 
\begin{equation}
f({\x}^k) \leq \frac{1}{k}\sum_{t=1}^k f(\tilde{\x}^t) = \frac{1}{k}\sum_{t=1}^k \sum_{i\in\myset{V}}\Big( q_i(\z_i^t) - \langle\z_i^t, \h_i(\tilde{\x}^t)\rangle\Big), \quad k\geq 1.
\end{equation} 
By Lemma~\ref{subgrad} inequality~\eqref{dummy5} with $\y = \z_i^t \in Z$,  
%
\begin{equation*}
q_i(\z_i^t) - q_i(\y^t) \leq \langle\h_i(\tilde{\x}^t), \z_i^t \rangle +  \langle\mathbold{\nu}, \z_i^t\rangle - \langle\h_i(\tilde{\x}^t) + \mathbold{\nu}, \y^t\rangle + \epsilon_t/n,
\end{equation*}
with $\epsilon_t/n = \beta_{t-1}(6 M + 3\tau) + \zeta$. Summing over $i\in\myset{V}$,
\begin{equation*}
\sum_{i\in\myset{V}} q_i(\z_i^t) \leq q(\y^t) + \sum_{i\in\myset{V}} \langle\h_i(\tilde{\x}^t), \z_i^t \rangle +  \sum_{i\in\myset{V}} \langle\mathbold{\nu}, \z_i^t\rangle - \langle\g(\tilde{\x}^t), \y^t\rangle + \epsilon_t, 
\end{equation*}
hence, 
\begin{equation}\label{dummy20}
f({\x}^k) \leq \frac{1}{k}\sum_{t=1}^k \Big( q(\y^t) + \sum_{i\in\myset{V}} \langle\mathbold{\nu}, \z_i^t \rangle- \langle\g(\tilde{\x}^t), \y^t \rangle+ \epsilon_t\Big).
\end{equation}
We can use Lemma~\ref{itiy}.\emph{(a)} with $\y = \mathbf{0} \in Z$ to upper bound $-\langle\g(\tilde{\x}^t), \y^t\rangle$, while we bound $\|\langle\mathbold{\nu},\z_i^t\rangle\|_2$ as  $\|\langle\mathbold{\nu}, \z_i^t\rangle\|_2 \leq \tau (\Lambda + \Gamma)$. The latter bound comes from the fact that by construction $\|\mathbold{\nu}\|_2\leq\tau$, and $\|\z_i^t\|_2\leq \Lambda + \Gamma$ by~\eqref{MM}. With this in place, we can write~\eqref{dummy20} as
\begin{equation*}
f({\x}^k) \leq \frac{1}{k}\sum_{t=1}^k q(\y^t) + n \tau (\Lambda + \Gamma) + \frac{\|\y^{1}\|^2_2}{2 k \alpha/n} + \frac{\alpha n(M+\tau)^2}{2} + \frac{1}{k}\sum_{t=1}^k\epsilon_t.
\end{equation*}
If we now compute
\begin{equation}\label{dummy30}
\frac{1}{k}\sum_{t=1}^k\epsilon_t = \frac{1}{k}\sum_{t=1}^k n (\beta_{t-1}(6M + 3\tau) + \zeta) \leq n (\beta_{0}(6M + 3\tau) + \zeta),
\end{equation}
and remember that by optimality $q(\y^t) \leq q^*$, $q^* = f^*$ by strong duality (Assumption~\ref{as.slater}), and $\|\y^1\|_2^2 \leq \Lambda^2+\Gamma^2$, then the claim follows. 
\qed
\end{proof}

\subsection{Proof of Theorem~\ref{th.3}.\emph{(b)}}

\begin{proof}
Given any dual optimal solution $\y^*$, we have
\begin{equation}\label{dummy6}
f(\x^k) = \underbrace{f(\x^k) + \Big\langle\y^*,\frac{1}{k}\sum_{t=1}^k \g(\tilde{\x}^t) \Big\rangle}_{(\mathrm{a})} - \Big\langle\y^*, \frac{1}{k}\sum_{t=1}^k \g(\tilde{\x}^t) \Big\rangle.
\end{equation}
\vskip-0.50cm \noindent We also know that, 
\begin{align}
(\mathrm{a}) &=  
f(\x^k) + \Big\langle\y^*, \frac{1}{k}\sum_{t=1}^k\sum_{i\in\myset{V}}\h_i(\tilde{\x}^t)\Big\rangle + n \Big\langle \y^* \frac{1}{k}\sum_{t=1}^k \d^t/\alpha \Big\rangle \nonumber \\ & \geq f(\x^k) + \Big\langle\y^*, \sum_{i\in\myset{V}}\h_i({\x}^k)\Big\rangle - n (\Lambda + \Gamma) \tau,\label{dummy76}
\end{align}
where we used the fact that $\h_i(\tilde{\x}^t)$ is a convex function of $\tilde{\x}^t$ and therefore,
$$
\frac{1}{k}\sum_{t=1}^k\sum_{i\in\myset{V}}\h_i(\tilde{\x}^t) \geq \sum_{i\in\myset{V}}\h_i({\x}^k),
$$
and the Cauchy-Schwarz inequality to bound 
$$
\Big\langle\y^*, \frac{1}{k}\sum_{t=1}^k \d^t/\alpha \Big\rangle \geq - \|\y^*\|_2 \Big\|\frac{1}{k}\sum_{t=1}^k \d^t/\alpha \Big\|_2 \geq - \tau (\Lambda+\Gamma) .
$$
Furthermore, by the saddle point property of the Lagrangian function, i.e., for any $\x \in X, \y \in Z$
$$
L(\x^*, \y) \leq L(\x^*, \y^*) \leq L(\x, \y^*), 
$$
and the fact that under strong duality (Assumption~\ref{as.slater}) $L(\x^*, \y^*) = q^* = f^*$, we can write
\begin{equation}\label{dummy7}
f(\x^k) + \Big\langle\y^*,\sum_{i\in\myset{V}}\h_i({\x}^k)\Big\rangle - n \tau(\Lambda + \Gamma) = L(\x^k, \y^*) - n \tau(\Lambda + \Gamma) \geq f^* - n \tau(\Lambda + \Gamma).
\end{equation}
We can now upper bound $\Big\langle\y^*,\frac{1}{k}\sum_{t=1}^k \g(\tilde{\x}^t) \Big\rangle$ in~\eqref{dummy6} as in Lemma~\ref{itiy}.\emph{(b)}, with $\y = 2 \y^* \in Z $ (by the definition of $r$). By substituting this bound in~\eqref{dummy6} and by combining it with~\eqref{dummy76} and~\eqref{dummy7}, we get
\begin{equation*}
f(\x^k) \geq f^* - n \tau (\Lambda + \Gamma) - \frac{\|\y^{1}-2\y^*\|^2_2}{2 k \alpha/n} - \frac{\alpha n(M+\tau)^2}{2} - \frac{1}{k}\sum_{t=1}^k\epsilon_t.
\end{equation*}
From the upper bound~\eqref{dummy30}, and $\|\y^1 - 2\y^*\|^2_2 = \|\y^1\|^2_2 + 4\|\y^1\|_2\|\y^*\|_2 + 4\|\y^*\|^2_2$, which can be upper bounded as $9(\Lambda^2+\Gamma^2)$, the claim follows. \qed
\end{proof}


\section{Numerical results}\label{sec:num}

In this section, we present some numerical results to assess the proposed algorithm for different $\varphi$ values in comparison with the standard dual decomposition. We choose the following simple yet representative sample problem,
\begin{equation*}
\minimize_{\small\begin{array}{c}x_i \in [0,1]\\ i\in\{1,\dots,100\} \end{array}} \hskip-0.3cm f(\x):= -\sum_{i=1}^{33} \sigma_i x_i -\sum_{i=34}^{100} \sigma_i \log(1+x_i),~\subjectto\,\sum_{i=1}^{100} \sigma_i x_i \leq 10 , 
\end{equation*}
where each $\sigma_i\in[0,1]$ is drawn from a uniform random distribution. This type of problem has been considered e.g. in network utility maximization contexts~\cite{Palomar2006}. We solve the problem in Matlab with Yalmip and SDPT3 \cite{Loefberg2004, Toh1999}, where we also implement the proposed algorithm\footnote{The code is available at: {http://ens.ewi.tudelft.nl/$\sim$asimonetto/NumericalExample.zip}.}. 

For this problem a Slater vector is $x_i = 0$ for all $i$; furthermore $\gamma = 10$, while $q(0)$ is solvable by inspection ($x_i = 1$) and gives (for our realization of $\sigma_i$) $r = 8.62$. The communication network is a randomly selected and the average number of neighbors is $3.12$.

Figure~\ref{fig.1} depicts convergence and it is in line with our theoretical findings: the error decreases as $O(1/k)$ till it reaches a bounded error floor. This bounded error floor depends on both $\varphi$ and $\alpha$ as captured in~Theorem~\ref{th.3}. We have also plotted the performance of the standard dual decomposition, which (in the absence of a master node), requires reaching complete consensus at each iteration (in theory $\varphi\to\infty$, but we have set $\varphi = 26$, which yields a full $\W^\varphi$). 

Figure~\ref{fig.2} shows the relative error with respect to the total number of messages the nodes are exchanging. We can see that, in the absence of a master node, the proposed consensus-based algorithm involves significantly fewer number of messages than the standard dual decomposition for the same accuracy level (till up to 1\% error). This is very important in real life applications.   

\begin{figure}
\centering
\psfrag{b}[c]{$|f^*-f(\x^k)|$}
\psfrag{c}[c]{Number of iterations $k$}
\psfrag{phiphiphiphiphie}{$\varphi = 1$, $\alpha = 1$}
\psfrag{phiphiphiphiphio}{$\varphi = 2$, $\alpha = 1$}
\psfrag{phiphiphiphiphis}{$\varphi = 4$, $\alpha = 1$}
\psfrag{phiphiphiphiphit}{$\varphi = 26$, $\alpha = 1$}
\psfrag{phiphiphiphiphik}{$\varphi = 1$, $\alpha = 0.1$}
\psfrag{rayline}{\hskip0.2cm $1/k$ line}
\includegraphics[width=0.7\textwidth]{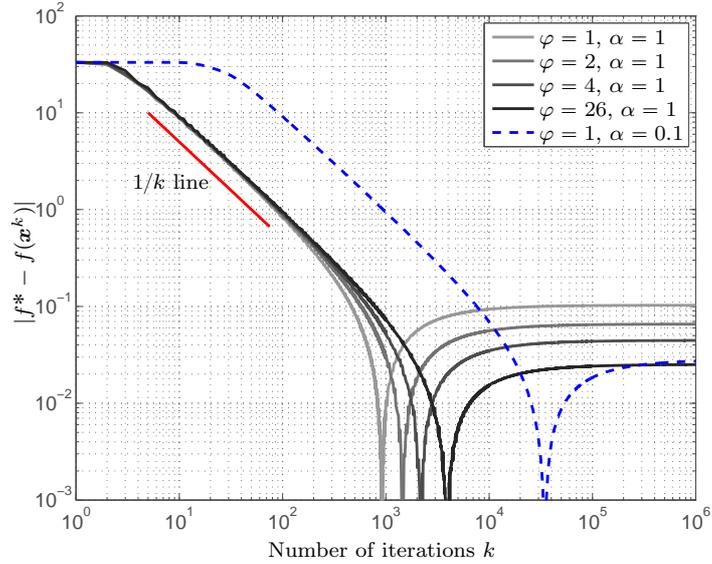}
\caption{Convergence of the proposed algorithm for different choices of stepsize $\alpha$ and number of consensus step $\varphi$.}
\label{fig.1}
\end{figure}

\begin{figure}
\centering
\psfrag{b}[c]{$(f^*-f(\x^k))/f^*$}
\psfrag{a}[c]{Total number of messages}
\psfrag{phiphiphiphiphie}{$\varphi = 1$, $\alpha = 1$}
\psfrag{phiphiphiphiphio}{$\varphi = 2$, $\alpha = 1$}
\psfrag{phiphiphiphiphis}{$\varphi = 4$, $\alpha = 1$}
\psfrag{phiphiphiphiphit}{$\varphi = 26$, $\alpha = 1$}
\psfrag{phiphiphiphiphik}{$\varphi = 1$, $\alpha = 0.1$}
\includegraphics[width=0.7\textwidth]{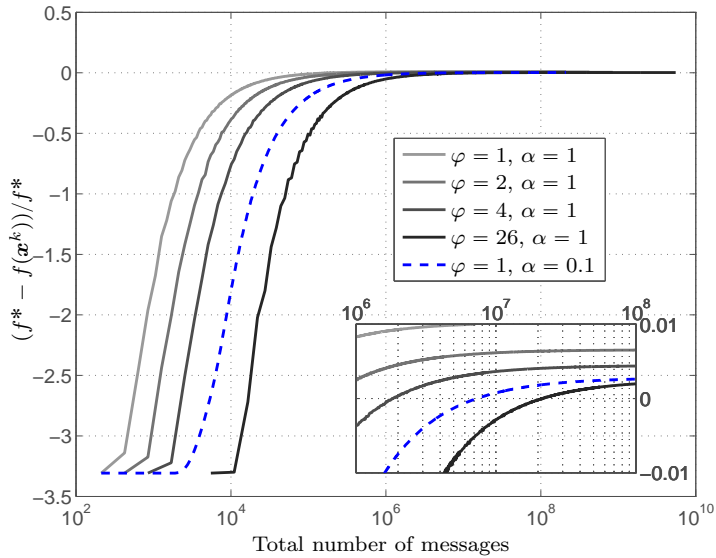}
\caption{Relative error and number of exchanged messages for different choices of stepsize $\alpha$ and number of consensus step $\varphi$.}
\label{fig.2}
\end{figure}

\section{Future research questions}\label{sec:rquestion}

Future research encompasses the following points. 

First of all, we have used the ergodic mean to recover the primal solution. The reason for it, is mainly technical: it helps to derive convergence rate results, via a telescopic cancellation argument. Other convex combinations have been advocated, e.g., in~\cite{Gustavsson2014}, but the results they can offer are typically asymptotical, and require vanishing stepsizes. An open question is whether other combinations for primal recovery are possible using constant stepsizes. 

Then, in the derivation, we have limited ourselves to objective convergence. It would be relevant to investigate convergence of the ergodic mean to the optimizer set, either in the general convex case or in the strong convex scenario. 

Finally, The bound on $\varphi$, i.e., $\bar{\varphi}$ has been derived in such a way that we could use $\epsilon$-subgradient arguments in the rest of the convergence proofs. However, it is quite conservative (in fact, in practice, $\varphi$ can be as small as $1$, but this is often not captured by the bound in Theorem~\ref{th.1}). This is due to Lemma~\ref{l.basic2} and the use of the spectral radius as an upper bound. A thorough investigation is left for future research.

\section{Conclusions}\label{sec:8}

A consensus-based dual decomposition scheme has been proposed to enable a network of collaborative computing nodes to generate approximate dual and primal solutions of a distributed convex optimization problem. We have proven convergence of the scheme both in the dual and the primal objective senses up to a bounded error floor. The proposed scheme is of theoretical and applied importance since it eliminates the need for a centralized entity (i.e., a master node) to collect the local subgradient information, by distributing this task among the nodes. This need has been a major hurdle in the use of dual decomposition for solving certain classes of distributed optimization problems. 

\footnotesize
\bibliographystyle{spmpsci}
\bibliography{../../PaperCollection2}

\end{document}